\documentclass[11pt,reqno]{amsart}
\usepackage{color}

\calclayout

\usepackage{amsmath,mathtools}
\usepackage{amsfonts,amscd}
\usepackage{amssymb}
\usepackage{url}
\usepackage{hyperref}

\usepackage[english]{babel}
\usepackage{booktabs}
\usepackage{tikz}

\theoremstyle{plain}
\newtheorem{theorem}                 {Theorem}      [section]

\newtheorem{corollary}    [theorem]  {Corollary}
\newtheorem{lemma}        [theorem]  {Lemma}
\newtheorem{proposition}  [theorem]  {Proposition}

\theoremstyle{definition}

\newtheorem{definition}   [theorem]  {Definition}

\newtheorem{example}      [theorem]  {Example}

\newtheorem{remark}       [theorem]  {Remark}

\usepackage{tikz}
\usetikzlibrary{matrix}
\usetikzlibrary{arrows.meta}
\usetikzlibrary{cd}

\numberwithin{equation}{section}

\newcommand{\defeq}{\vcentcolon=}

\def \cn{{\mathbb C}}
\def \C{{\mathbb C}}

\def \rn{{\mathbb R}}
\def \R{{\mathbb R}}

\def \zn{{\mathbb Z}}

\def \E{\mathcal E}
\def \F{\mathcal F}

\def\nab#1#2{\hbox{$\nabla$\kern -.3em\lower 1.0 ex
		\hbox{$#1$}\kern -.1 em {$#2$}}}
\def\hatnab#1#2{\hbox{$\nabla$\kern -.3em\lower 1.0 ex
		\hbox{$#1$}\kern -.1 em {$#2$}}}

\def\span{\mathrm{span}}

\def \SL2{\widetilde{\text{\bf SL}}_{2}(\rn)}

\def \SO#1{\mathbf{SO}(#1)}

\def \U#1{\text{\bf U}(#1)}

\def \SU#1{\text{\bf SU}(#1)}

\def \Sp#1{\text{\bf Sp}(#1)}

\DeclareMathOperator{\Div}{div} 

\DeclareMathOperator{\trace}{trace}

\usepackage{enumitem}
\setlist[enumerate,1]{label={(\roman*)}}

\numberwithin{equation}{section}
\allowdisplaybreaks

\begin{document}

	

\thanks{OR gratefully acknowledges the support of the Deutsche Forschungsgemeinschaft (DFG, German Research Foundation) - Project-ID 427320536 - SFB 1442, as well as of Germany's Excellence Strategy EXC 2044 390685587, Mathematics M\"unster: Dynamics-Geometry-Structure. TM gratefully acknowledges the hospitality of Mathematics M\"unster during his visit.}

\author{Thomas Jack Munn}
\address{Mathematics, Faculty of Science\\
	University of Lund\\
	Box 118, Lund 221 00\\
	Sweden}
\email{Thomas.Munn@math.lu.se}

\author{Oskar Riedler}
\address{Universit\"at M\"unster, Mathematisches Institut, Einsteinstr. 62, 48149 M\"unster, Germany}
\email{oskar.riedler@uni-muenster.de}

\title
[]
{$(\lambda,\lambda)$-Eigenfunctions on Compact Manifolds}

\begin{abstract}
In this note we study $(\lambda,\mu)$-eigenfamilies on compact Riemannian manifolds when $\lambda = \mu$. We show that any compact manifold admitting a $(\lambda,\lambda)$-eigenfunction is a mapping torus and that any $(\lambda,\lambda)$-eigenfamily is one dimensional. Additionally, we consider generalised eigenfamilies, which can have higher dimension, and relate these to harmonic Riemannian submersions to a torus.\end{abstract}

\maketitle

\section{Introduction}
\label{section-introduction}

Let $(M,g)$ be a Riemannian manifold, $\F = \{ \phi_i: M \to \cn \mid i \in I \}$ a family of complex-valued functions, and $\lambda,\mu\in\cn$. We say that $\F$ is a \emph{$(\lambda,\mu)$-eigenfamily}, if for every $i,j \in I$ we have that
$$ \tau(\phi_i) = \lambda \cdot \phi_i, \qquad g(\nabla \phi_i, \nabla \phi_j) = \mu \cdot \phi_i \phi_j,$$
where $\tau$ is the Laplace-Beltrami operator and $g$ denotes the complex bilinear extension of the metric to the complexified tangent bundle. Elements of $(\lambda,\mu)$-eigenfamilies are called \emph{$(\lambda,\mu)$-eigenfunctions}. 
\smallskip

Eigenfamilies were introduced by Gudmundsson and Sakovich \cite{Gud-Sak-1} in order to produce harmonic morphisms. Since then they have been utilised to generate other geometrically interesting data, such as $r$-harmonic functions \cite{Gud-Sob-1},  and minimal submanifolds of codimension two in \cite{Gud-Mun-1} by Gudmundsson and the first named author.

Eigenfamilies have been shown to exist on all the classical Riemannian symmetric spaces, see \cite{Gud-Sak-1, Gud-Sak-2, Gud-Sif-Sob-2, Gha-Gud-4, Gha-Gud-5}. 
The well studied case of  $(\lambda,\mu)$-eigenfunctions when $(\lambda,\mu)=(0,0)$ is that of \emph{harmonic morphisms with codomain $\cn$}. For the general theory of harmonic morphisms between Riemannian manifolds we refer to the excellent book \cite{Bai-Woo-book} and the regularly updated online bibliography \cite{Gud-bib}.
\smallskip

Recently, there has been interest in classifying eigenfamilies. $(\lambda,\mu)$-eigenfamilies on $m$-spheres are shown to correspond to $(0,0)$-eigenfamilies of homogeneous polynomials on the ambient space $\rn^{m+1}$ in \cite{Rie-1} by the second named author. When $M$ is \emph{compact} many results describing the spectra, existence, and topological properties of eigenfunctions are derived in \cite{Rie-Sif-1} by Siffert and the second named author. In particular, they show that eigenfamilies on compact manifolds have $\lambda\leq\mu<0$ and that the case $\lambda=\mu$ is rather singular: the image degenerates to a circle and this case must be excluded from certain constructions. 

\smallbreak
We begin this note by providing a complete characterisation of $(\lambda,\lambda)$-eigenfamilies on compact manifolds. The situation for a single eigenfunction is as follows:
\begin{theorem} \label{thm-ll-eigenfunction}
Let $(M,g)$ be a compact and connected Riemannian manifold, $\phi:M\to\cn$ a non-constant smooth map, $\lambda<0$. The following are equivalent:
\begin{enumerate}
\item $\phi$ is a $(\lambda,\lambda)$-eigenfunction.
\item For any $x_0\in M$ the map $\pi: (M,g)\to (S^1, \frac1{|\lambda|} dt^2)$, $x\mapsto \frac{\phi(x)}{|\phi(x_0)|}$ is a well-defined harmonic Riemannian submersion.
\end{enumerate}
\end{theorem}
\smallbreak

Harmonic Riemannian submersions to $S^1$ are well understood. The only Riemannian manifolds admitting Riemannian submersions to a circle are mapping tori, i.e. manifolds which may be written as
$$(M,g)\cong \left(\frac{M_0 \times [0,2\pi]}{(x,0)\sim (\eta(x),2\pi)},\ g(t)+\frac1{|\lambda|}dt^2\right)$$
where $M_0$ is compact, $g(t)$ a Riemannian metric on $M_0$ for all $t$, and $\eta:M_0\to M_0$ is a diffeomorphism with $\eta^*g(2\pi)= g(0)$. The Riemannian submersion $[x,t]\mapsto e^{it}$ is then harmonic if and only if the fibres are minimal, i.e. if and only if the volume density of $g(t)$ on $M_0$ is constant in $t$.

The above remarks can be considered to be a version of Tischler's Theorem \cite{Tis} in the Riemannian setting. We formulate and prove this as Theorem \ref{thm-circle-bundle} in Section \ref{sec: 3} below. 

Beyond this we show that any $(\lambda,\lambda)$-eigenfamily on a compact and connected $(M,g)$ must necessarily be spanned by a single function, whence the previous theorem already gives the complete picture for $(\lambda,\lambda)$-eigenfamilies:

\begin{proposition} \label{Theorem-Span-LambdaLambda}
Let $(M,g)$ be a compact and connected Riemannian manifold, $\lambda\in\cn$ and $\mathcal F$ a $(\lambda,\lambda)$-eigenfamily on $M$. Then $\dim(\span_\cn(\mathcal F)) \leq 1$.
\end{proposition}
	
In certain situations a generalised notion of a $(\lambda,\mu)$-eigenfamily appears, where one allows different values of $\lambda$ for each function, and the number $\mu$ is replaced by a matrix. To be precise:
\begin{definition}
Let $\mathcal F=\{\phi_1,...,\phi_k\}$ be a finite family of functions $M\to\cn$ and $\lambda_i, A_{ij}\in\cn$ for $i,j\in\{1,...,k\}$. We say that $\mathcal F$ is $(\lambda_i, A_{ij})$-eigenfamily if
$$\tau(\phi_i) = \lambda_i\, \phi_i,\qquad g(\nabla \phi_i,\nabla\phi_j) = A_{ij}\, \phi_i \phi_j$$
for all $i,j\in\{1,...,k\}$.
\end{definition}
Examples of these families appear on Sasaki manifolds, where one can build an embedding 
$$(\phi_1,...,\phi_k):(M,g)\to \cn^k$$
for which $\{\phi_1,...,\phi_k\}$ is a $(\lambda_i, -\sqrt{\mu_i \mu_j})$-eigenfamily for appropriate values $\lambda_i$, $\mu_i$. The construction of these families is carried out in Theorem 3 of \cite{Rie-Sif-1}.
\begin{example}
Let $\mathbf{w}=(w_1,...,w_n)\in\rn_{>0}^n$, we denote the associated weighted Sasakian sphere by $S^{2n-1}_{\mathbf w}$. As a Riemannian manifold this is the unit sphere $S^{2n-1}\subset \cn^n$ equipped with the metric
$$g_{\mathbf w}(v,v) = \frac1{\eta(\xi_{\mathbf w})}\left( \|v\|^2-2\langle\xi_{\mathbf w}, v\rangle \eta_{\mathbf w}(\xi_{\mathbf w})\eta_{\mathbf w}(v)\right)+(1+\frac{\|\xi_{\mathbf w}\|^2}{\eta(\xi_{\mathbf w})} )\eta_{\mathbf w}(v)^2$$
where $v\in TS^{2n-1}$ and $\xi_{\mathbf w}=\sum_i w_i (x_i\partial_{y_i}-y_i\partial_{x_i})$, $\eta=\sum_i (x_i d y_i-y_idx_i)$, $\eta_{\mathbf w}= \tfrac{1}{\eta(\xi_{\mathbf w})} \eta$ for $(x_1+iy_1,...,x_n+iy_n)$ the standard complex coordinates on $\cn^n$. Here $\langle,\rangle$, $\|\cdot\|$ denote the usual Euclidean scalar product and norm. Then the family $\{\phi_1=(x_1+iy_1),...,\phi_n=(x_n+iy_n)\}$ is a $(\lambda_i, A_{ij})$-eigenfamily on $S^{2n-1}_{\mathbf w}$ for
$$\lambda_i = - w_i^2 - w_i(2n-2) ,\qquad A_{ij}= - w_i w_j.$$\end{example}

These kinds of families can still be used to construct local harmonic morphisms:
\begin{theorem} \label{Thm-Generalised-EF-Holo}
Let $(M,g)$ be a Riemannian manifold, $\lambda_i\in \cn, A_{ij}\in \cn$ for $i,j \in \{1,\dots,k \}$, and $\{\phi_1,...,\phi_k\}$ a $(\lambda_i, A_{ij})$-eigenfamily on $(M,g)$. For $U\subseteq \cn^k$ open and $F:U\to\cn$ holomorphic with:
$$F(t_1z_1,...,t_k z_k) = (\prod_it_i^{d_i})F(z_1,...,z_k)$$
for all $(z_1,...,z_k)\in U$ and $t_i\in\rn$ close enough to $1$. Then
$$(\phi_1,...,\phi_k)^{-1}(U)\to\cn, \qquad x\mapsto F(\phi_1(x),...,\phi_k(x))$$
is a $(\widetilde\lambda,\widetilde\mu)$-eigenfunction for $\widetilde\lambda = \sum_i d_i(\lambda_i - A_{ii}) + \sum_{ij} d_i d_j A_{ij}$, $\widetilde\mu= \sum_{ij} d_i d_j A_{ij}$. In particular if
$$\sum_{ij} d_i d_j A_{ij}=0, \qquad \sum_i d_i (\lambda_i - A_{ii})=0,$$
then the map is a harmonic morphism.

\end{theorem}

\begin{example}
Let $F:\cn^k\to\cn, (z_1,...,z_k)\mapsto z_1^{d_1}\cdots z_k^{d_k}$ be a monomial map. Then $F$ satisfies the conditions of Theorem \ref{Thm-Generalised-EF-Holo}, and so $F\circ (\phi_1,...,\phi_k)$ gives a new globally defined eigenfunction for any $(\lambda_i, A_{ij})$-eigenfamily on a manifold $M$.\end{example}

In Section \ref{sec: 4} we prove Theorem \ref{torus}, which states that a generalised eigenfamily $\mathcal{F}$ consisting of $(-A_{ii},-A_{ij})$-eigenfunctions on $M$ is equivalent to the existence of a harmonic Riemannian submersion $\pi:(M,g) \to (T^k,A^{-1})$.
After this, we conclude the paper in Section \ref{sec: 5} by applying the characterisations of Theorems \ref{thm-ll-eigenfunction} and \ref{torus} in order to derive existence and non-existence results for such families from curvature conditions.
\smallbreak

\section{Eigenfunctions and Eigenfamilies}
\label{section-eigenfunctions}
In this section we remind the reader of some important definitions and results that we will use later.
Let $(M,g)$ be an $m$-dimensional Riemannian manifold and $T^{\cn}M$ be the complexification of the tangent bundle $TM$ of $M$. We extend the metric $g$ to a complex bilinear form on $T^{\cn}M$.  Then the gradient $\nabla\phi$ of a complex-valued function $\phi:(M,g)\to\cn$ is a section of $T^{\cn}M$.  In this situation, the well-known complex linear {\it Laplace-Beltrami operator} (alt. {\it tension field}) $\tau$ on $(M,g)$ acts locally on $\phi$ as follows
$$
\tau(\phi)=\Div (\nabla \phi)=\sum_{i,j=1}^m\frac{1}{\sqrt{|g|}} \frac{\partial}{\partial x_j}
\left(g^{ij}\, \sqrt{|g|}\, \frac{\partial \phi}{\partial x_i}\right).
$$
For two complex-valued functions $\phi,\psi:(M,g)\to\cn$ we have the following well-known fundamental relation
\begin{equation*}\label{equation-basic}
\tau(\phi\, \psi)=\tau(\phi)\,\psi +2\,\kappa(\phi,\psi)+\phi\,\tau(\psi),
\end{equation*}
where the complex bilinear {\it conformality operator} $\kappa$ is given by $$\kappa(\phi,\psi)=g(\nabla \phi,\nabla \psi).$$  In local coordinates it takes the form:
$$\kappa(\phi,\psi)=\sum_{i,j=1}^mg^{ij}\,\frac{\partial\phi}{\partial x_i}\frac{\partial \psi}{\partial x_j}.$$

\begin{definition}\cite{Gud-Sak-1}\label{definition-eigenfamily}
Let $(M,g)$ be a Riemannian manifold, $\lambda,\mu\in\cn$. Then a complex-valued function $\phi:M\to\cn$ is said to be a \emph{$(\lambda,\mu)$-eigenfunction} if it is eigen both with respect to the Laplace-Beltrami operator $\tau$ and the conformality operator $\kappa$ with respective eigenvalues $\lambda,\mu$, i.e.
$$\tau(\phi)=\lambda\cdot\phi\ \ \text{and}\ \ \kappa(\phi,\phi)=\mu\cdot \phi^2.$$	
A set $\E$ of complex-valued functions on $M$ is said to be a \emph{$(\lambda,\mu)$-eigenfamily} on $M$ if for all $\phi,\psi\in\E$ we have 
$$\tau(\phi)=\lambda\cdot\phi\ \ \text{and}\ \ \kappa(\phi,\psi)=\mu\cdot \phi\,\psi.$$ 
\end{definition}

Eigenfamilies can be used to produce complex-valued harmonic morphisms:
\begin{theorem}[\cite{Gud-Sak-1}] \label{THM-Gud-Sak-HM-From-EF}
Let $(M,g)$ be a semi-Riemannian manifold, $\lambda, \mu\in\cn$, and 
$$ \F = \{ \phi_1, \dots ,\phi_n\}$$
a $(\lambda,\mu)$-eigenfamily. If $P,Q: \cn^n \to \cn$ are linearly independent homogeneous polynomials of the same positive degree then the quotient
$$ \frac{P(\phi_1,\dots, \phi_n)}{Q(\phi_1,\dots, \phi_n)}$$
is a non-constant harmonic morphism on the open and dense subset
$$\{ p \in M \, | \, Q(\phi_1(p), \dots,\phi_n(p)) \neq 0\}.$$
\end{theorem}

Away from its zeros, one can write a complex function in polar form, in this case the condition of being an eigenfunction becomes as follows:
\begin{lemma}[\cite{Rie-Sif-1}] \label{Rie-Sif-ef-polar}
Let $(U,g)$ be a Riemannian manifold, not necessarily compact or complete, and let $\phi:U\to \cn$ be a $(\lambda,\mu)$-eigenfunction with $\lambda,\mu$ both real and $\phi(x)\neq0$ for all $x\in U$. Suppose $\phi(x)=e^{i \vartheta(x)}|\phi(x)|$ for some smooth function $\vartheta:U\to\rn$. Then:
\begin{enumerate}
\item $\tau(\vartheta)=0$;
\item $\tau(\ln|\phi|)=\lambda-\mu$;
\item $\kappa(\vartheta,|\phi|)=0$;
\item $\kappa(\ln|\phi|,\ln|\phi|)=\kappa(\vartheta,\vartheta)+\mu$.
\end{enumerate}
\end{lemma}

The case of a $(\lambda,\lambda)$-eigenfunction is special. One has the following characterisation:
\begin{proposition}[\cite{Rie-Sif-1}] \label{Rie-Sif-lambdalambda}
Let $(M,g)$ be a compact connected Riemannian manifold and $\phi:M\to\cn$ a non-constant $(\lambda,\mu)$-eigenfunction. The following are equivalent:
\begin{enumerate}
\item $\lambda=\mu$.
\item $|\phi|^2$ is constant.
\item $\phi(x)\neq0$ for all $x\in M$.
\end{enumerate}
\end{proposition}

\section{$(\lambda,\lambda)$-eigenfunctions}\label{sec: 3}
Theorem \ref{thm-ll-eigenfunction} is a special case of Theorem \ref{torus}, which will be proven in the next section. We begin this section by proving Proposition \ref{Theorem-Span-LambdaLambda}, showing that all $(\lambda,\lambda)$-eigenfamilies on compact connected manifolds are necessarily one-dimensional:

\begin{proof}[Proof of Proposition \ref{Theorem-Span-LambdaLambda}]
Suppose that $\phi,\psi$ are non-constant and in the same $(\lambda,\lambda)$-eigenfamily $\F$. Then it follows from Theorem \ref{THM-Gud-Sak-HM-From-EF} that the quotient
$$ \frac{\phi}{\psi}: M\setminus\psi^{-1}(\{0\}) \to \cn$$
is a harmonic morphism, in particular it is a harmonic map. Since $\psi(x) \neq 0$ for all $x\in M$ by Proposition \ref{Rie-Sif-lambdalambda}, the domain of $\tfrac{\phi}{\psi}$ is all of $M$. By compactness of $M$ it must then be constant, i.e. $\psi$ and $\phi$ are linearly dependent.
\end{proof}

The following result characterises compact manifolds which admit $(\lambda,\lambda)$-eigenfunctions.

\begin{theorem}\label{thm-circle-bundle}
	Let $(M,g)$ be compact and connected, $\lambda<0$, and $\pi:(M,g)\to (S^1,\frac1{|\lambda|} dt^2)$ a smooth map. The following are equivalent:
	\begin{enumerate}[label=(\roman*)]
		\item The map $\pi$ is a harmonic Riemannian submersion.
		\item $M$ is a mapping torus $$M_0 \times_{\eta} [0,2\pi] =\frac{M_0 \times [0,2\pi]}{(x,0) \sim (\eta(x),2\pi)}$$ with metric
					$$ g= g(t)+ \frac1{|\lambda|}\,dt^2$$
					and monodromy map $\eta:M_0 \to M_0$ with $\eta^*g(2\pi)=g(0)$, $M_0$ is compact, the volume density of $g(t)$ is constant in $t$, and $\pi([x,t])=e^{it}$.
	\end{enumerate}

\end{theorem}
\begin{proof}
$(i) \Rightarrow (ii)$.
Since $\pi:M\to S^1$ is a surjective submersion from a compact manifold it is a proper map, and thus by Ehresmann's Theorem it is a fibre bundle over $S^1$. We let $M_0$ denote a typical fibre, and $X=\nabla\hat\pi$, where $\hat\pi$ is the locally defined lift of $\pi$ with co-domain $\rn$.

Note that $X$ is a horizontal vector field, it is divergence free since $\hat\pi$ is harmonic. Additionally from $\hat\pi$ being a Riemannian submersion:
$$d\hat\pi(X)= g(X,X)= \frac{d\hat\pi (X)^2}{|\lambda|},$$
whence $d\hat\pi(X)=\|X\|^2=-\lambda$ is constant. Letting $\eta_t$ denote the flow of $X$ one then sees that
$$\partial_t (\hat\pi\eta_t(x))= d\hat\pi (X) = -\lambda,$$
and so $\hat\pi(\eta_t(x))= -\lambda \cdot t+\hat\pi(x)$, i.e. $\pi(\eta_t(x))=e^{it}\pi(x)$. It follows that $M$ is a mapping torus
$$M\cong \frac{M_0\times[0,2\pi]}{(x,0)\sim(\eta(x),2\pi)}$$
with monodromy map $\eta=\eta_{2\pi}$, and $\pi([x,t])=e^{it}$. Under the above diffeomorphism $X\equiv \partial_t$, since $X$ was horizontal one finds that the metric on $M$ is of the form
$$g=g(t)+\tfrac1{|\lambda|}dt^2$$
where $\eta^*g(2\pi)=g(0)$ and $g(t)$ is a family of metrics on the fibre $M_0$.

We now show that each $g(t)$ induces the same volume density. Assume first that $M$ is oriented, let $\omega(t)$ denote the volume form of $g(t)$ so that $\omega(t)\wedge dt$ is the volume form of $M$ at $[x,t]$. Then
$$(\partial_t \omega(t)) \wedge dt = (\mathcal L_{X} \omega(t))\wedge dt = \mathcal L_{X}(\omega(t)\wedge dt) = \mathrm{div}(X)\, \omega(t)\wedge dt = 0,$$
where we used the formula $\mathcal L_{X} d\mathrm{Vol} = \mathrm{div}(X)$ and $\mathrm{div}(X)=0$. Since the above argument is entirely local it also shows that the volume density induced by $g(t)$ is constant in $t$ for the non-orientable case.

\bigbreak

\bigbreak

$(ii) \Rightarrow (i)$. Note that the map
$$\widehat \pi:( M_0 \times \rn, g(t) +dt^2)\to (\rn, \cdot),\qquad (x,t)\mapsto t$$
is clearly a Riemannian submersion. We now verify that it is harmonic. Let $(x_0,t_0)\in M_0 \times \rn$ and let $x_1,...,x_m$ be a coordinate system of $M_0$ so that $\partial_{x_1},...,\partial_{x_m}$ is an orthonormal basis of $T_{x_0}M_0$ with respect to $g(t_0)$. We denote $g_{ij}(t) = g(t)\,(\partial_{x_i},\partial_{x_j})$, note that $g(t_0)_{ij}=\delta_{ij}$. One checks that $\nabla\widehat\pi= \partial_t$, from which it follows that
\begin{eqnarray*}
	\Delta \widehat\pi &=& \sum_{i=1}^m g(t_0)\,(\nabla_{\partial_{x_i}} \partial_t, \partial_{x_i}) = \sum_{i=0}^m g(t_0)\,(\nabla_{\partial_t} \partial_{x_i}, \partial_{x_i}) = \frac12\frac d{dt}\sum_{i=0}^m g_{ii}(t)\lvert_{t=t_0}\\
	&=&\frac12\frac d{dt}\trace( (g_{ij})_{1\leq i,j\leq m})\lvert_{t=0} = \frac12 \frac d{dt} \ln\left(\det( (g_{ij}(t))_{1\leq i,j\leq m})\right)\lvert_{t=t_0}.
\end{eqnarray*}

This last expression is $0$, since the volume form $\sqrt{\det((g_{ij})_{1\leq i,j\leq m})} dx_1 \wedge \cdots \wedge dx_m$ is constant in the variable $t$ by assumption, and so $\widehat \pi$ is harmonic.

Now the diagram
\begin{center}
	\begin{tikzcd}
		(M_0 \times \rn, g(t)+dt^2)  \arrow[r,"\widehat\pi"] \arrow[d,"/\zn"]    & (\rn, \cdot) \arrow[d,"/\zn"] \\
		(M_0\times_\eta I, g(t)+dt^2) \arrow[r,"\pi"]  & (S^1, dt^2)
	\end{tikzcd}
\end{center}
commutes by construction, here the vertical arrows are the natural covering maps. Since the vertical maps are also local isometries it follows that $\pi$ is a harmonic Riemannian submersion.
\end{proof}

\section{Generalised Eigenfamilies}\label{sec: 4}

The previous section shows that the notion of a $(\lambda,\mu)$-eigenfamily becomes degenerate in the case $\lambda=\mu$, since these are then always spanned by a single function. By relaxing some of the conditions on $(\lambda,\mu)$-eigenfamilies the $(\lambda,\lambda)$-eigenfunctions can be part of a family:

\begin{definition}
Let $\mathcal F=\{\phi_1,...,\phi_k\}$ be a finite family of functions $M\to\cn$ and $\lambda_i\in\cn$, $1\leq i\leq k$ a vector in $\cn^k$ and $A_{ij}$, $1\leq i,j \leq k$ a symmetric complex $k\times k$ matrix.
\begin{enumerate}[label=(\roman*)]
\item We call $\mathcal F$ an \emph{$(\lambda_i,A_{ij})$-eigenfamily} if for all $\phi_i,\phi_j\in\mathcal F$:
$$\kappa(\phi_i,\phi_j) = A_{ij} \phi_i \phi_j, \qquad \tau(\phi_i) = \lambda_i \phi_i.$$
\item We say that the family is \emph{$\lambda$-diagonal} if additionally $\lambda_i=A_{ii}$ for all $i$.
\item $\mathcal F$ is said to be \textit{reduced} if $A_{ij}$ is a non-degenerate matrix.
\end{enumerate}
\end{definition}

The following is the prototypical example of a $\lambda$-diagonal $(\lambda_i,A_{ij})$-eigenfamily.

\begin{example}[\cite{Rie-Sif-1}]
Let $\Gamma$ be a lattice in $\rn^n$ and $M= \rn^n / \Gamma$ be the corresponding flat torus. Then consider the dual lattice
$$ \Gamma^* = \{ k \in \rn^n \, \mid \, \langle k, k' \rangle \in \mathbb{Z}, \ \text{for all } k \in \Gamma\}.$$

For any subset $K \subset \Gamma^*$ we define the family
$$ \mathcal{F} = \{f_k(x) = e^{2 \pi i \langle k,x \rangle} \, \mid \, k \in K\}.$$

A direct computation shows 
$$ \tau(f_k) = - 4 \pi^2 \langle k,k \rangle \cdot f_k, \ \ \ \kappa(f_k,f_{k'}) =- 4 \pi^2 \langle k,k' \rangle \cdot f_k f_{k'}$$
and thus $\mathcal{F}$ is a $\lambda$-diagonal eigenfamily. Furthermore the set $K$ is linearly independent if and only if $\mathcal{F}$ is reduced.
\end{example}

\begin{proof} [Proof of Theorem \ref{Thm-Generalised-EF-Holo}]
Let $F:U\to\cn$ as in the statement of the theorem. The homogeneity condition means that $z_i\cdot \partial_{z_i} F (z_1,...,z_k)= d_i F$ for all $i$. Writing $^\wedge$ to denote precomposition with $(\phi_1,...,\phi_k)$ one then finds:
$$\kappa(\hat F, \hat F) = \sum_{ij}(\partial_{z_i}F)^\wedge\,(\partial_{z_j} F)^\wedge\, \kappa(\phi_i,\phi_j)= \sum_{ij} A_{ij} \, (z_i \partial_{z_i}F\,z_j \partial_{z_j}F)^\wedge =\sum_{ij}A_{ij}d_id_j (\hat F)^2.$$
Similarly:
$$\Delta \hat F =\sum_i (\partial_{z_i} F)^\wedge\, \Delta\phi_i+\sum_{ij}(\partial_{z_i}\partial_{z_j} F)^\wedge \cdot \kappa(\phi_i,\phi_j) =\sum_i d_i (\lambda_i - A_{ii})\hat F + \sum_{i j}A_{ij} d_i d_j \hat F.$$
Completing the proof.
\end{proof}

We now turn to $\lambda$-diagonal eigenfamilies:

\begin{proposition}\label{prop: A-definite}
Let $(M,g)$ be compact and connected, $ A\in M_{k\times k}(\cn)$, and $\mathcal F=\{\phi_1,...,\phi_k\}$ a $(A_{ii},A_{ij})$-eigenfamily on $M$, in particular $\mathcal F$ is $\lambda$-diagonal. Then:
\begin{enumerate}[label=(\roman*)]
\item $A$ is a negative semi-definite real matrix.
\item $\mathcal F$ is not reduced if and only if there exist $\alpha_i\in\R$ not all vanishing so that $\prod_{i=1}^k (\phi_i)^{\alpha_i}$ is constant.
\end{enumerate}
\end{proposition}
\begin{proof}
Since $\mathcal F$ is $\lambda$-diagonal $|\phi_i(x)|$ is constant on $M$ for all $i$. Then the locally defined functions $\ln\phi_i$ are purely imaginary and the vector fields $\nabla \vartheta_i\defeq\nabla\mathrm{Im} (\ln \phi_i)=\frac{\nabla\phi_i}{i\phi_i}$ satisfy:
$$g(\nabla \vartheta_i,\nabla \vartheta_j) =-\frac{A_{ij}\phi_i\phi_j}{\phi_i\phi_j}= -A_{ij},$$
whence $A_{ij}$ is negative semi-definite. Further it is degenerate if and only if the $\nabla \vartheta_i$ are linearly dependent, i.e. iff $\sum_i \alpha_i \nabla \vartheta_i=0$ for some $\alpha_i\in\R$ not all vanishing. This is equivalent to
$$0=\sum_i \alpha_i \nabla \ln(\phi_i) = \nabla \ln \prod_i (\phi_i)^{\alpha_i},$$
which completes the proof.
\end{proof}

\begin{definition}
For $A$ a positive definite $k\times k$ matrix, let $(T^k, A^{-1})$ denote the flat torus $T^k=(S^1)^k$ equipped with metric $A^{-1}$.
\end{definition}

\begin{theorem} \label{Thm-Eigenfamily-HM} \label{torus}
Let $(M,g)$ be a compact Riemannian manifold, $A$ a symmetric real $k\times k$ matrix, and $\mathcal F=\{\phi_1,...,\phi_k\}$ a family of functions $M\to\C$. The following are equivalent:
\begin{enumerate}
\item $\mathcal F$ is a reduced $(-A_{ii}, -A_{ij})$-eigenfamily.
\item $A$ is positive definite and for any $x_0\in M$ the map $$\pi: (M,g)\to (T^k, A^{-1}), \ \ \ x\mapsto (\frac{\phi_1(x)}{|\phi_1(x_0)|},...,\frac{\phi_k(x)}{|\phi_k(x_0)|})$$ is a well-defined and harmonic Riemannian submersion.
\end{enumerate}
\end{theorem}
\begin{proof}
(i)$\implies$(ii): Since $\phi_i$ is a $(-A_{ii},-A_{ii})$-eigenfunction for all $i$ one has by Proposition \ref{Rie-Sif-lambdalambda} that $|\phi_i(x)|$ is constant and non-zero for all $i$, and so $\pi:M\to T^k, x\mapsto (\tfrac{\phi_1(x)}{|\phi_1(x_0)|},...,\tfrac{\phi_k(x)}{|\phi_k(x_0)|})$ is well defined.

Additionally by Lemma \ref{Rie-Sif-ef-polar} one finds around each point in $M$ a neighbourhood $U$ and maps $\vartheta_i: U\to \rn$ so that
$$\left(\tfrac{\phi_1(x)}{|\phi_1(x_0)|},...,\tfrac{\phi_k(x)}{|\phi_k(x_0)|}\right) = \left(e^{i\vartheta_1(x)},...,e^{i \vartheta_k(x)}\right).$$
One verifies:
\begin{equation}
g( \nabla \vartheta_i, \nabla \vartheta_j) = - g(\nabla \ln\phi_i,\nabla \ln\phi_j)= A_{ij},\qquad \tau(\vartheta_i)=0.\label{eq: hihj}
\end{equation}
Since $\mathcal F$ is reduced Proposition \ref{prop: A-definite} implies that $A$ is a positive definite matrix. Equation (\ref{eq: hihj}) then implies that the map $\widetilde\pi:(U,g)\to(\rn^k, A^{-1})$, $x\mapsto (\vartheta_1(x),...,\vartheta_k(x))$ is a harmonic Riemannian submersion to its image.

Letting $L:(\rn^k,A^{-1})\to (T^k,A^{-1})$, $(t_1,...,t_k)\mapsto (e^{it_1},...,e^{it_k})$ denote the standard isometric covering map, one has by construction that $\pi\lvert_{U}=L\circ\widetilde\pi$, whence $\pi$ is a harmonic Riemannian submersion.\smallskip

(ii)$\implies$(i): Since we assume that $\pi: M \to T^k$, $x\mapsto (\tfrac{\phi_1(x)}{|\phi_1(x_0)|},...,\tfrac{\phi_k(x)}{|\phi_k(x_0)|})$ is well defined it follows that $|\phi_i(x)|$ is constant non-zero for all $i$. We then find locally defined functions $\vartheta_i\defeq \mathrm{Im}(\ln\phi_i)$ so that
$$\pi(x)= (e^{i\vartheta_1(x)},...,e^{i\vartheta_k(x)})= (L\circ \widetilde\pi)(x).$$
Here we let $L:(\rn^k,A^{-1})\to (T^k,A^{-1})$, $(t_1,...,t_k)\mapsto (e^{it_1},...,e^{it_k})$ denote the standard isometric covering map. The map $\widetilde\pi(x)=(\vartheta_1(x),...,\vartheta_k(x))$ can be defined locally around any point.

Note that:
$$D\widetilde\pi (\nabla \vartheta_i) = \sum_j d\vartheta_j(\nabla \vartheta_i)\,\partial_{t_j} = \sum_j g(\nabla \vartheta_i, \nabla \vartheta_j)\,\partial_{t_j}.$$
Since $\pi$ is Riemannian submersion so is $\widetilde\pi$, and it follows that:
$$g(\nabla \vartheta_i, \nabla \vartheta_j)=A^{-1}(D\widetilde\pi\nabla \vartheta_i, D\widetilde\pi\nabla \vartheta_j)=\sum_{k,\ell} g(\nabla \vartheta_i,\nabla \vartheta_k)A^{-1}_{k\ell}g(\nabla \vartheta_\ell,\nabla \vartheta_j),$$
which implies
$$g(\nabla \phi_i,\phi_j)= -\phi_i\phi_j\,g(\nabla \vartheta_i, \nabla \vartheta_j) =- A_{ij}\phi_i\phi_j.$$
Additionally $\widetilde\pi$ is harmonic since $\pi$ is, and so
$$\tau(\phi_i)=-\phi_i\, g(\nabla \vartheta_i,\nabla \vartheta_i) +i\,\phi_i\tau(\vartheta_i) = -A_{ii}\phi.$$
It follows that $\mathcal F$ is a $(-A_{ii}, -A_{ij})$-eigenfamily, since $A$ is positive definite it is reduced by Proposition \ref{prop: A-definite}.
\end{proof}

\begin{remark}
Theorem \ref{torus} shows that any compact manifold $(M,g)$ admitting a reduced $\lambda$-diagonal $(-A_{ii}, -A_{ij})$-eigenfamily is a fibre bundle over the torus $T^k$. If, additionally, the horizontal distribution is integrable then one can proceed as in Theorem \ref{thm-circle-bundle} to show that $M$ is a quotient of $$M_0 \times [0,2\pi]^k $$ with metric $g(t_1, \dots t_{n-k})+A^{-1}$ such that $g(t_1, \dots t_{n-k})$ induces the a constant volume density on $M_0$.\end{remark}


\section{Existence and Non-existence from Curvature}\label{sec: 5}
In this section we apply our characterisations of $(\lambda,\lambda)$-eigenfunctions and $\lambda$-diagonal eigenfunctions by harmonic Riemannian submersions in order to derive existence and non-existence results.

We begin by remarking that harmonic maps $\pi:M\to S^1$ correspond, up to rotations of $S^1$, with harmonic $1$-forms on $M$ with integral periods (see e.g. \cite{Bai-Woo-book}, \cite{Helein}). The condition that the resulting map is submersive is  equivalent to the form being \textit{nowhere vanishing}. Theorems \ref{thm-ll-eigenfunction} and \ref{torus} then imply:

\begin{corollary}Let $(M,g)$ be a compact and connected Riemannian manifold.\begin{enumerate}[label=(\roman*).]
	\item Suppose $Ric \geq 0$ and first Betti number $\beta_1(M) = k$. Then there exists a reduced $\lambda$-diagonal eigenfamily $\mathcal{F} = \{ \phi_1, \dots, \phi_k\} $.
	\item Suppose the sectional curvature of $(M,g)$ is strictly negative, then $(M,g)$ does not admit any non-constant $(\lambda,\lambda)$-eigenfunctions for any value of $\lambda$.
	\end{enumerate}
\end{corollary}
\begin{proof}
	We prove (i) first. Since $M$ is compact and $\beta_1(M)=k$, there exist linearly independent integral harmonic $1$-forms $\omega_1,\dots \omega_k$. Let $\phi_1,\dots \phi_k$ denote the corresponding harmonic maps into $S^1$.
	
	Since $Ric \geq 0$ and $M$ is compact, we have that the $\omega_i$ are parallel by Bochner \cite{Bochner} and thus
	$$ g(\omega_i,\omega_j) = g(  d\phi_i,  d\phi_j) $$
	is constant and so the resulting map $(\phi_1 ,\dots, \phi_k)$ is a harmonic Riemannian submersion to a torus.
	
	Point (ii) follows from the fact, due to Tsagas \cite{Tsa}, that every closed $1$-form on a compact manifold with strictly negative curvature vanishes at at least one point.\end{proof}



\end{document}